\documentclass[12pt,a4paper]{article}
\usepackage[utf8]{inputenc}
\usepackage{enumitem}
\usepackage[english]{babel}
\usepackage{latexsym,amssymb}
\usepackage{xcolor}
\usepackage{tensor}

\usepackage{cite}
\usepackage[pdftex]{graphicx}
\usepackage{tikz}
\usepackage{amsmath,amsthm,amsfonts,amssymb,bbm}
\usepackage{graphicx,psfrag,subfigure,float,color}
\usepackage{cite}
\usepackage{graphicx}
\usetikzlibrary{arrows}
\usepgflibrary{snakes}
\usetikzlibrary{arrows,snakes,backgrounds}
\DeclareGraphicsExtensions{.jpg,.pdf,.pdftex,.eps}

\newcommand{\Pb}{\mathbb{P}}

\newcommand{\dx}{\mathrm{d}}
\newcommand{\R}{\mathbb{R}}
\newcommand{\N}{\mathbb{N}}
\newcommand{\Z}{\mathbb{Z}}

\newcommand{\GUE}{\mathrm{GUE}}
\newcommand{\GOE}{\mathrm{GOE}}

\newtheorem{tthm}{Theorem}

\newtheorem{prop}{Proposition}[section]

\newtheorem{lem}[prop]{Lemma}
\newtheorem{defin}[prop]{Definition}
\newtheorem{cor}{Corollary}

\newtheorem{rem}[prop]{Remark}

\theoremstyle{definition}

\newcommand{\blocktheorem}[1]{%
  \csletcs{old#1}{#1}
  \csletcs{endold#1}{end#1}
  \RenewDocumentEnvironment{#1}{o}
    {\par\addvspace{1.5ex}
     \noindent\begin{minipage}{\textwidth}
     \IfNoValueTF{##1}
       {\csuse{old#1}}
       {\csuse{old#1}[##1]}}
    {\csuse{endold#1}
     \end{minipage}
     \par\addvspace{1.5ex}}
}

\usepackage{graphicx}
\usepackage{amssymb}
\usepackage{epstopdf}
\usepackage{nicefrac}
\author{P. Nejjar\thanks{Institute for Applied Mathematics, Bonn University, Endenicher Allee 60, 53115 Bonn, Germany. E-mail: {\tt nejjar@iam.uni-bonn.de}. This work is supported by the Deutsche Forschungsgemeinschaft
(German Research Foundation) by the CRC 1060 (Projektnummer
211504053) and Germany's Excellence Strategy - GZ 2047/1, Projekt ID
390685813. 
}}
\begin{document}

\title{Dynamical Phase Transition of ASEP in the KPZ Regime}

\date{}

\maketitle 
\begin{abstract}
We consider the asymmetric simple exclusion process (ASEP) on $\Z$. For continuous densities,  ASEP is in local equilibrium for large times, 
at discontinuities however,  one expects to see a dynamical phase transition, i.e. a mixture  of  different  equilibriums. 
We consider ASEP with  deterministic initial data such that at large times, two rarefactions come together at the origin, and the density jumps from $0$ to $1$. Shifting the measure on the KPZ $1/3$ scale, we show that the law of  ASEP converges to a mixture of the Dirac measures with only holes resp. only particles. The parameter of that mixture is the probability that the second class particle, which is distributed as the difference of two independent GUEs,  stays to the left of the shift. This should be compared with the  results of Ferrari and Fontes  from 1994 \cite{FF94b}, who obtained a mixture of Bernoulli product measures at  discontinuities created by random initial data, with the GUEs replaced by Gaussians.
\end{abstract}
\section{Introduction}
For large times, interacting particle systems are expected to be in local equilibrium.    
However, local equilibrium does not hold when the density of particles is discontinuous. Rather, one expects to observe  what Wick \cite{W85} called a dynamical phase transition, a mixture of different  equilibriums. The aim of this paper is to study an interacting particle system - the asymmetric simple exclusion process (ASEP) on $\Z$ - in a situation where local equilibrium does not hold, and, at the same time, the Kardar-Parisi-Zhang (KPZ) behavior of ASEP can be observed because no initial randomness supersedes it. We refer to \cite{BoGo15} for an introduction to integrable probability and KPZ universality.

ASEP can be described as follows: Each site $i\in \Z$ is occupied either by a particle or a hole. The particles perform independent, continuous-time random walks, waiting a mean $1$  exponential time to make a unit step to the right with probability $p>1/2$ or  a unit step to the left with probability $q=1-p<1/2$. However, the step is only made if the target site is occupied by a hole, and when a step is made, the particle and the hole exchange positions. Equivalently, we can think of the holes as performing random walks, jumping to the left (resp. right) with probability $p$ (resp. $q$), and being only allowed to jump to sites occupied by particles. This is the particle-hole duality. The state space of ASEP is $\Omega=\{0,1\}^{\Z},$ the $1'$s are considered particles, the $0'$s are considered holes.  

The evolution of ASEP is known to be closely related to the Burgers equation for $u(\xi,\theta)\in \R$ (where $\xi,\theta\in \R$) given by
\begin{equation*}
\partial_{\theta} u+(p-q)\partial_{\xi}[u(1-u)]=0.
\end{equation*}
Indeed, let $\zeta^{N},N\in \N,$ be a sequence in $\Omega$ such that  the initial empirical density satisfies
\begin{equation}
\lim_{N\to \infty} \frac{1}{N}\sum_{i\in \Z}\zeta^{N}(i)\delta_{i/N}=u(\xi,0) \mathrm{d}\xi,
\end{equation}
where $\delta_{i/N}$ is the Dirac measure at $i/N$ and the convergence is in the sense of vague convergence of measures. 
Denoting by $\zeta^{N}_{t}$ the state of the ASEP started from $\zeta^{N}$ at time $t,$ we have at later times 
\begin{equation}
\lim_{N\to \infty} \frac{1}{N}\sum_{i\in \Z}\zeta^{N}_{\theta N}(i)\delta_{i/N}=u(\xi,\theta) \mathrm{d}\xi,
\end{equation}
where $u(\xi,\theta)$ is the unique entropy solution of the Burgers equation with initial data $u(\xi,0)$, see \cite{BM06}. 

Closely related to this is  the convergence of the law of $\zeta^{N}_{N\theta},$ which we will denote by $\delta_{\zeta^{N}}S(N\theta),$ as  a measure on $\Omega$.  The set of possible limits, i.e. the invariant measures, were described completely in  \cite{Lig76}, they are the closed convex hull of the extremal invariant measures, which are given by 
\begin{equation}
\{\nu_{\rho},\rho\in [0,1]\}\cup \{\mu_{Z}, Z\in \Z\}.
\end{equation}
Here $\nu_{\rho},\rho\in [0,1],$ are the product Bernoulli measure on $\Omega$ under which $\zeta(j),j\in \Z,$ are i.i.d. random variables and $\Pb(\zeta(j)=1)=1-\Pb(\zeta(j)=0)=\rho.$
Note that for $\rho\in\{0,1\}$, $\nu_\rho$ is a Dirac measure, respectively  on the configuration without particles and the configuration without holes. The $\mu_{Z}$ are conditional blocking measures and defined in \eqref{muZ} below.

ASEP is in local equilibrium at all macroscopic times $\theta>0$ whenever $u(\cdot,\theta)$ is continuous: As shown in \cite[Theorem 2]{BM06} in a more general setting, 
at every continuity point $\xi_{0}$ of $u(\cdot,\theta),$   $\zeta^{N}_{\theta N}\tau_{\xi_{0}N}$ converges in distribution to $\nu_{u(\xi_0, \theta)}.$ Here  the shift operator $\tau_{n}$ acts on $\zeta\in \Omega$ by $\zeta \tau_n = \zeta (\cdot+n)$ and naturally extends to measures on $\Omega$.

This local equilibrium does not hold when there is a discontinuity (shock). 
For shocks between regions of constant densities $\rho<\lambda$ created by random initial data,   \cite[Theorem 1.3]{FF94b} showed that  the limit law of ASEP  at this shock   is a convex combination of $\nu_{\rho}$ and $\nu_{\lambda},$ rather than a single product measure.  
Shifting the measure on the scale of the fluctuations of the second class particle at the shock, the parameter of this convex combination is precisely the probability that the second class particle stays to the left of this shift. This probability is given by a Gaussian, and the  Gaussian comes from the random initial data, not ASEP itself.

This naturally leads to the question what happens in the absence of initial randomness, where one expects to see the (non-Gaussian) KPZ behavior of ASEP.  As in our previous work \cite{N20CMP}, here we consider a shock between two rarefaction fans, which meet at the origin where the density jumps from $0$ to $1,$ see Figure \ref{one}. We find, see Theorem \ref{Main}, that again there is a dynamical phase transition, namely we obtain a convex combination of the Bernoulli/Dirac measures $\nu_{0}$ and $\nu_{1}$. Shifting the measure on the  scale of the fluctuations of the second class particle at the shock, we find that the parameter of the convex combination is again the probability that the second class particle stays to the left of the shift. In our case however, the second class particle is governed by the KPZ $1/3$ fluctuation exponent, and it is  distributed as the difference of two independent $\GUE$s (see \eqref{FGUE2}), rather than  the difference of two independent Gaussians.
This is the first example of a dynamical phase transition of ASEP in the KPZ regime, even for $p=1$ (TASEP), we are not aware of such a result. Furthermore, in Theorem \ref{muhat} we obtain the limit law of the process as seen from the second class particle. 
 Our proofs  develop further the methods employed in \cite{N20CMP}, we refer to Section \ref{heur} for a description of our proof method and describe now our main results. 

\begin{figure}
 \begin{center}
   \begin{tikzpicture}
       \draw (0.4,0.7) node[anchor=south]{\small{$1$}};
  \draw [very thick, ->] (0,-0.5) -- (0,2.5);
    \draw (-0.1,2) node[anchor=east] {\small{$u(\xi,0)$}};
    \draw[red,thick ] (-2,1.3) -- (-1,1.3);
      \draw[thick,red] (0,1.3) -- (1,1.3);
\draw[very thick] (-1,-0.1)--(-1,0.1)  ;
  \draw (-1,-0.6) node[anchor=south]{\small{$q-p$}};
  \draw[very thick] (1,-0.1)--(1,0.1)  ;
  \draw (1,-0.6) node[anchor=south]{\small{$p-q$}};
\filldraw(0,1.3) circle(0.08cm);
   \draw [very thick, ->] (-2,0) -- (2,0) node[below=4pt] {\small{$\xi$}};

\begin{scope}[xshift=7cm]
 \draw [very thick,->] (-2,0) -- (2,0) node[below=4pt] {\small{$\xi$}};
  \draw [very thick,->] (0,-0.5) -- (0,2.5);
    \draw (0.3,1.3) node[anchor=south]{\small{$1$}};
    \draw[red,thick] (-2,1.3) -- (0,0);
    \draw[red,thick] (0,1.3)--(2,0);
    \draw[very thick] (-1,-0.1) -- (-1,0.1);
       \draw (-0.8,-0.6) node[anchor=south]{\small{$q-p$}};
    \draw[very thick] (1,-0.1) -- (1,0.1);
       \draw (1.2,-0.6) node[anchor=south]{\small{$p-q$}};
\filldraw(0,1.3) circle(0.08cm); 
    \draw (-0.1,2) node[anchor=east] {\small{$u(\xi,1)$}};
   \end{scope}
   \end{tikzpicture}    \end{center} \caption{ Left: The   initial particle density $u(\xi,0)$ of    the initial configuration $\eta^{1}$.
     Right: The large time  density $u(\xi,\theta)$ at the macroscopic time  $\theta=1$. We can informally think of the parameter $M,$ while invisible on the hydrodynamic scale,  as allowing us to transit between the fluctuations at $\theta=1$ and those at  $\theta>1.$
      }\label{one}
\end{figure}
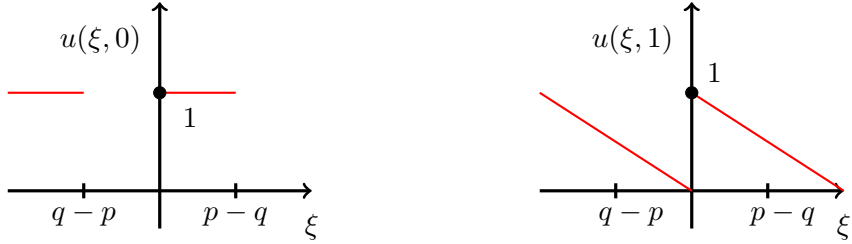

The initial data we will consider is 
\begin{equation}\label{eta1}
\eta^{1}=\mathbf{1}_{\Z_{< -N(t,M)}}+\mathbf{1}_{\{0,\ldots,N(t,M)\}},
\end{equation}
 where $N(t,M)$ is given by a parameter $M\in \Z_{\geq 1}$ and 
\begin{equation*}
C(M):=2\sqrt{\frac{M}{p-q}}, \quad \quad N(t,M):=(p-q)(t-C(M)t^{1/2}).
\end{equation*}

So really, $\eta^{1}$ is a sequence of initial data which depends on $M$ and $t,$ when  we want to emphasize this, we may write $\eta^{1}=:\eta^{1,M,t},$ but we will often  omit the $M,t$ to lighten our notation.  
For each fixed $M$, we will look at the process $(\eta^{1,M,t}_{\ell},0\leq \ell \leq t) $ 
and then send $t\to\infty.$ The configuration $\eta^{1}$ has an infinite group of particles starting from $\Z_{<-N(t,M)},$ and another group of particles starting from  
$\{0,\ldots,N(t,M)\}.$ Each group will form a region of decreasing density (rarefaction fan), and at  time $t,$ these two fans meet at the origin for the first time, see Figure \ref{one}. The parameter $M$ tunes the interaction between the two rarefaction fans: As $M$ becomes large, more and more particles that started in $\Z_{<-N(t,M)}$ arrive at the origin, and more and more holes that started in $\Z_{>N(t,M)}$
arrive at the origin also.  Alternatively, we may remove the parameter $M$ from the initial configuration completely,  start an ASEP from 
\begin{equation*}
\hat{\eta}^{1}=\mathbf{1}_{\Z_{< -(p-q)\ell}}+\mathbf{1}_{\{0,\ldots,(p-q)\ell\}},
\end{equation*}
and study the ASEP started from $\hat{\eta}^{1}$ at time $t=\ell+C(M)\ell^{1/2}$. From this viewpoint, sending $M\to \infty$ after sending  $t\to \infty$ can be heuristically  interpreted as transitioning  from the fluctuations at macroscopic time $\theta=1$ to the fluctuations at  macroscopic time $\theta>1,$ see also Remark \ref{rem} in Section \ref{mainproof}.  In the following, we will work with $\eta^{1},$ not $\hat{\eta}^{1}$.

To formulate our  dynamical phase transition, we define the Tracy-Widom $\GUE$ distribution function, which originates in random matrix theory
 \cite{TW94}, as 
\begin{equation}\label{FGUE2}
F_{\GUE}(s)=\sum_{n=0}^{\infty} \frac{(-1)^{n}}{n!}  \int_{s}^{\infty}\dx x_{1}\ldots  \int_{s}^{\infty}\dx x_{n}\det(K_{2}(x_{i},x_{j})_{1\leq i,j\leq n}),\end{equation}
where $K_{2}(x,y)$ is the Airy kernel  $K_{2}(x,y)=\frac{Ai(x)Ai^{\prime}(y)-Ai(y)Ai^{\prime}(x)}{x-y},x\neq y, $ defined for $x=y$ by continuity and $Ai$ is the Airy function.  

 Let for $\xi \in \R$
\begin{equation*}
p(\xi):=\Pb(\chi_{\GUE}-\chi_{\GUE}^{\prime}\leq \xi),
\end{equation*}
where $\chi_{\GUE},\chi_{\GUE}^{\prime}$ are two independent, $\GUE-$distributed random variables. Note $p(0)=1/2$. Recall the shift operator $\tau_{n}$ acting on $\Omega$ as $\zeta\tau_{n}=\zeta(\cdot+n)$ and naturally extending to subsets of $\Omega$ and measures on $\Omega$.
The dynamical phase transition is as follows. Recall that the initial configuration $\eta^{1}$ defined in \eqref{eta1} depends also on $M$ and $t$.
\begin{tthm}\label{Main}Denote by $\delta_{\eta^{1}}S(t)$ the law of $\eta^{1}_{t}$. We have for $\xi \in \R$
\begin{equation*}
\lim_{M\to\infty}\lim_{t\to\infty}\delta_{\eta^{1}}S(t)\tau_{M^{1/3}\xi}=(1-p(\xi))\nu_{0}+p(\xi)\nu_{1},
\end{equation*}
where both limits are in the sense of weak convergence of measures.
\end{tthm}
Theorem \ref{Main} is proven in Section \ref{mainproof}. In fact, in Theorem \ref{MainPrep} of Section \ref{mainproof}, we show that   $\lim_{t\to\infty}\delta_{\eta^{1}}S(t)$  equals an infinite linear combination of the blocking measures $\mu_{Z}$ defined in \eqref{muZ} below. A heuristical derivation of this infinite linear combination and Theorem \ref{Main} is given in  Section \ref{heur}.
 
\subsection{Comparison with  shocks created by random initial data}

Let $\nu_{\rho,\lambda}$ be the product measure on $\Omega$ for which each nonnegative integer is occupied by a particle with probability $\lambda$ and each negative integer  is occupied by a particle with probability $\rho$, where $\lambda>\rho,(\rho,\lambda)\neq (0,1).$
In this case, the  initial density equals $\rho\mathbf{1}_{\R_{-}}+\lambda\mathbf{1}_{\R_{+}}$,   the shock moves with speed $v=(p-q)(1-\lambda-\rho)$ 
and exhibits a dynamical phase transition \cite[Theorem 1.3]{FF94b}.

To compare this shock better to our situation, we can send simultaneously $\rho\to 0, \lambda\to 1:$ Let $\rho=\varepsilon, \lambda=1-\varepsilon$.
It is then an immediate corollary of \cite[Theorem 1.3]{FF94b} that with $D(\varepsilon)=\frac{2\varepsilon(1-\varepsilon)}{1-2\varepsilon}$ we have 
\begin{equation}\label{Gaussian}
\lim_{\varepsilon \to 0}\lim_{t\to\infty} \nu_{\varepsilon,1-\varepsilon}S(t) \tau_{vt+\xi(D(\varepsilon)t)^{1/2}}=(1-\widehat{p}(\xi))\nu_{0}+\widehat{p}(\xi)\nu_{1},
\end{equation}
where
\begin{equation*}
\widehat{p}(\xi)=\frac{1}{\sqrt{2\pi}}\int_{-\infty}^{\xi}\mathrm{d}s e^{-s^{2}/2}
\end{equation*}
is the standard Gaussian. This Gaussian is obtained as the difference of two independent Gaussians, see \cite[Theorem 1.1]{FF94b} and Section \ref{scp}.
 
Theorem \ref{Main}, as well as the dynamical phase transition \eqref{Gaussian}, is best  understood via the behaviour of the second class particle at the shock, which we describe next.

\subsection{Second class particles}\label{scp}
To define the second class particle, we consider  the  configuration $\eta^{2}$, which is obtained by replacing the particle $\eta^{1}$ has at the origin by a hole:
\begin{equation}
\eta^{2}(j)=\eta^{1}(j)\mathbf{1}_{\Z\setminus\{0\}}(j).
\end{equation}
We couple the ASEPs starting from $\eta^{1},\eta^{2}$ via the basic coupling. This allows us to define the second class particle $X(t)$ as the position where these two ASEPs differ, i.e. 
\begin{equation}\label{X}
X(t)=\sum_{j\in \Z}j\mathbf{1}_{\{\eta^{1}_{t}(j)\neq \eta^{2}_{t}(j) \}}.
\end{equation}
The second class particle  interacts with holes like a particle, and with particles like a hole, and we will only consider ASEPs with a single second class particle.
Considering the enlarged state space $\{0,1,2\}^{\Z},$ where a $2$ indicates the presence of the second class particle, we can define the initial configuration $\eta$ with a second class particle at the origin as 
\begin{equation}\label{eta}
\eta(j)=\eta^{2}(j)+2\mathbf{1}_{\{0\}}(j).
\end{equation}
The  behavior of $X(t)$ in the double limit $ \lim_{M\to \infty}\lim_{t \to \infty}$ was obtained previously in  our  work \cite{N20CMP}. Note $M^{1/3} \xi$ is exactly the shift in Theorem \ref{Main}.
\begin{tthm}[Theorem 1 of \cite{N20CMP}]\label{MAIN}
For $\xi \in \R$ we have
\begin{equation*}
\lim_{M\to \infty}\lim_{t \to \infty}\Pb\left(X(t)\leq M^{1/3} \xi\right)=p(\xi).
\end{equation*}
\end{tthm}
In \cite{N20CMP}, we obtained an upper and a lower bound for $\lim_{t \to \infty}\Pb\left(\frac{X(t)}{M^{1/3}}\leq \xi\right)$ which both converge to $p(\xi)$ as $M\to\infty$. Our refined approach here allows us to compute $\lim_{t \to \infty}\Pb\left(X(t)\leq i\right),i\in \Z,$ explicitly, see Theorem \ref{X(t)} in Section \ref{seven}.

Let $\eta^{\rho,\lambda,1}$ be the particle configuration which has a particle  at the origin and  for which each $j>0$ is occupied by a particle with probability $\lambda,$ each $j<0$ is occupied by a particle  with probability $\rho$, and the occupation of different integers happens independently.
Setting $\eta^{\rho,\lambda,2}(j)=\eta^{\rho,\lambda,1}\mathbf{1}_{\Z\setminus\{0\}}(j),$ we can define as in \eqref{X} a second class particle $X^{\rho,\lambda}(t)$  starting at the origin.  As shown in \cite[Theorem 1.2]{FF94b},  $X^{\rho,\lambda}(t)$ converges to a Gaussian for arbitrary $\rho<\lambda,(\rho,\lambda)\neq (0,1).$  In the situation of \eqref{Gaussian}, \cite[Theorem 1.2]{FF94b} yields the following central limit theorem:
\begin{equation}
\lim_{\varepsilon\to 0}\lim_{t\to \infty}\Pb\left(X^{\varepsilon,1-\varepsilon}(t)\leq vt+(D(\varepsilon)t)^{\frac{1}{2}}\xi\right)=\widehat{p}(\xi).
\end{equation}
This Gaussian really is the difference of two independent Gaussians, since, as shown in \cite[Theorem 1.1]{FF94b}, $X^{\rho,\lambda}(t)$ is proportional to  the difference of  the number of particles present initially  in a region of length $\mathcal{O}(t)$ and the number of holes present initially in another region of length $\mathcal{O}(t)$.
Both these numbers being the sum of i.i.d. Bernoulli random variables, $X^{\rho,\lambda}(t)$ is asymptotically distributed as the difference of two independent Gaussians, i.e. as a Gaussian.

 Following the notation of \cite{Fer90}, we denote by 
$\hat{\nu}_{\rho,\lambda}\hat{S}(t)$
 the law of $\eta^{\rho,\lambda,2}_{t}\tau_{X^{\rho,\lambda}(t)},$ which is the process as seen from the second class particle. 
For a measure $\mu$ on $ \Omega,$ we write $
\mu(f)=\int_{\Omega}\dx \mu f
$
  and for $\lambda,\rho\in [0,1]$ we write
$\mu \sim \nu_{\rho,\lambda}$ if for all cylindric $f$ on $\Omega$ we have 
\begin{equation*}
\lim_{x\to-\infty}\mu \tau_{x}(f)=\nu_{\rho}(f) \quad \mathrm{and}\quad \lim_{x\to\infty}\mu \tau_{x}(f)=\nu_{\lambda}(f).
\end{equation*} 
Ferrari \cite{Fer90} proved the following general result, which shows that $X^{\rho,\lambda}$ is a microscopic shock, for us the special case $(\rho,\lambda)=(0,1)$ will be relevant.
\begin{tthm}[Theorem 2.2 of \cite{Fer90}]\label{2.2} Let $0\leq \rho<\lambda\leq 1$.
Uniformly in $t,$ we have $\hat{\nu}_{\rho,\lambda}\hat{S}(t)\sim \nu_{\rho,\lambda}.$  As $t\to\infty,$  $\hat{\nu}_{\rho,\lambda}\hat{S}(t)$ converges weakly to an invariant  measure $\hat{\mu}\sim \nu_{\rho,\lambda}. $ 
\end{tthm}
Note that for $\lambda=1,\rho=0,$ there is initially a second class particle at the origin,  which has only particles to its right, and only holes to its left. 
While our second class particle $X(t)$ starts from a quite different environment, it turns out that the process as seen from $X(t)$  converges to the same limit:  

\begin{tthm}\label{muhat}
The law of $\eta^{2}_{t}\tau_{X(t)}$ converges weakly, as $t\to \infty$,  to the measure $\hat{\mu}$ from Theorem \ref{2.2} with $\lambda=1,\rho=0$.
\end{tthm}
Theorem \ref{muhat} is proven in Section \ref{seven}, a heuristical derivation  is given in Section \ref{heur}.
Note that for every finite $t,$ $X(t)$  has  $\mathcal{O}(t)$ many  particles (resp. holes) to its right (resp. left),  in contrast to the property $\hat{\nu}_{0,1}\hat{S}(t)\sim \nu_{0,1}.$ 
In Corollary \ref{Xcor} of Section \ref{seven}, we show that after the $t\to\infty$ limit, the density of particles to the  right (resp. left) of $X(t)$ approaches $1$ (resp. $0$) exponentially fast. 

In light of Theorem \ref{MAIN} and Theorem \ref{muhat}, we can interpret Theorem \ref{Main} as follows: If the second class particle $X(t)$  stays to the left of $M^{1/3}\xi$, which happens with probability $p(\xi)$, then we see only particles  in the limit, i.e. $\nu_{1}$,  whereas if it stays to the right of $M^{1/3}\xi,$  which happens with probability $1-p(\xi),$ we see only holes, i.e. $\nu_{0}$.

\subsection{Outline}
In Section \ref{2}, we recall positive recurrent ASEPs and describe our method of proof. In Section \ref{3}, we introduce key random variables and events which will allow us in Section \ref{propsec} to reduce our problem  to positive recurrent ASEPs. In Section \ref{mainproof}, we prove Theorem \ref{Main}, and in Section \ref{seven} we prove Theorem \ref{muhat}, as well as Corollary \ref{Xcor} on the exponential convergence of densities and Theorem \ref{X(t)} giving the $t\to\infty$ limit law of $X(t)$.

\section{Positive recurrent ASEPs and method of proof}\label{2}

Recall that a sequence of measures $(\sigma_{n})_{n \geq 1}$  on $\Omega$ converges weakly to a measure $\sigma$ on  $\Omega$ if for all continuous $f$ 
we have that $ \sigma_{n} (f)$ converges to $\sigma(f). $ 
We will use the following simple statement to prove weak convergence.
\begin{prop}\label{convprop}
For $A\subseteq\Z$, denote $f_A=\mathbf{1}_{\{\zeta\in \Omega:\zeta(i)=1 \, \forall i\in A\}}$.
We have that $\lim_{n\to\infty}\sigma_{n}=\sigma$ in the sense of weak convergence of meaures  if for all finite subsets $A\subset\Z$ we have that
$\lim_{n\to\infty}\sigma_{n}(f_A)=\sigma(f_A).$
\end{prop}
Proposition \ref{convprop} can be shown by writing the configurations $\eta^{\zeta}$ defined in \cite[page 22]{Li85b} with $\eta =\mathbf{1}_\emptyset$ as finite linear combinations of functions of the form $f_A$.

Note that $f_\emptyset (\zeta)=1$ for all $\zeta\in \Omega$.

A well-known time-invariant measure $\mu$ for ASEP is the product blocking measure (see \cite[page 211]{Li99}) given by 
\begin{equation}
\mu(\{\zeta:\zeta(i)=1\})=\frac{(p/q)^{i}}{1+ (p/q)^{i}},
\end{equation}
for which the particle density approaches $1$ (resp. $0$) exponentially fast as $i\to +\infty$ (resp. $i\to -\infty$).  The measure $\mu$ thus concentrates on configurations $\zeta\in \Omega$ for which
\begin{equation}
\sum_{j<0}\zeta(j) <\infty \quad \sum_{j>0}(1-\zeta(j)) < \infty.
\end{equation}
The set of such configurations is given by the countable, disjoint  union of $\Omega_{Z}, Z\in \Z,$ where 
\begin{equation}\label{omega}
\Omega_{Z}=\left\{\zeta:\sum_{j<Z}\zeta(j)=\sum_{j\geq Z}(1-\zeta(j))<\infty\right\}.
\end{equation}
An element of $\Omega_Z$ that will appear later is  the reversed step initial data given by 
\begin{equation*}
\eta^{-\mathrm{step}(Z)}=\mathbf{1}_{\Z_{\geq Z}}.
\end{equation*}
When restricted to $\Omega_{Z}$, ASEP is an irreducible,  positive recurrent, countable state space Markov chain with unique  stationary measure \begin{equation}\label{muZ}\mu_{Z}:=\mu(\cdot|\Omega_{Z}),\end{equation} 
see \cite[page 212]{Li99} and the reference therein. In particular, we have the following result.
\begin{prop}\label{posrec}Let $\zeta\in \Omega_{Z}$. Then, the law of $\zeta_{t}$ converges to $\mu_{Z}$ as $t\to\infty.$
\end{prop}

Finally, we note that 
there is a partial order on $\Omega_{Z}$: For $\eta^{\prime}, \eta^{\prime\prime}$, we define 
\begin{equation}\label{order}
\eta^{\prime}\preceq\eta^{\prime\prime}\iff \sum_{j=r}^{\infty}(1-\eta^{\prime\prime }(j))\leq \sum_{j=r}^{\infty}(1-\eta^{\prime}(j))\quad \mathrm{for \, all\,}r\in \Z.
\end{equation}
It is easy to see that under the basic coupling, this order is preserved, i.e.  if $\eta^{\prime}\preceq\eta^{\prime\prime},$ then also $\eta^{\prime}_{t}\preceq\eta^{\prime\prime}_{t},t\geq 0.$ Note that  $\eta^{-\mathrm{step}(Z)}$ is maximal in $\Omega_{Z}$ w.r.t. the order $\preceq$.

\subsection{Method of proof}\label{heur}
Let us outline here the method of proof for   Theorem \ref{Main} and Theorem \ref{muhat}, which develops further and improves the methods described  in \cite[Section 1.2]{N20CMP}.  In particular, here we are able to show convergence of the entire ASEP and the process as seen from the second class particle,  as well as  to obtain single $t\to\infty$ limits explicitly  and not just  double limits $M\to\infty, \,t\to\infty$. 

Note that in the initial data $\eta^{1}$ from \eqref{eta1}, there is an infinite group of particles starting from $\Z_{<-N(t,M)},$ and an infinite group of holes starting from $\Z_{>N(t,M)}$. Let $\mathcal{P}(t)$ (resp. $\mathcal{H}(t)$) be the random number of particles (resp. holes) from these two groups which at time $t$ are sufficiently close to the origin to affect $\delta_{\eta^{1}}S(t)$ in the $t\to \infty$ limit. The random variables $\mathcal{P}(t),\mathcal{H}(t)$ will be properly defined in \eqref{PH}.
Let $x_{\mathcal{P}(t)}$ be the leftmost particle that started in   $\Z_{<-N(t,M)},$
and is close to the origin at time $t$, and likewise  let $ H_{\mathcal{H}(t)}$ be the rightmost hole  that started from  $\Z_{>N(t,M)}$ and is close to the origin at time $t$.
Let $\mathcal{T}\leq t$ be the random first time at which  $ H_{\mathcal{H}(t)}$ and $x_{\mathcal{P}(t)}$ have arrived close to the origin.
After time  $\mathcal{T}$,  no new particles arrive from $\Z_{<N(t,M)}$ to affect
$\delta_{\eta^{1}}S(t),$ and no new holes 
arrive from $ \Z_{>N(t,M)}$  either. So after time $\mathcal{T},$ we may replace all particles to the left of $x_{\mathcal{P}(t)}$ by holes, and all holes to the right  of 
$ H_{\mathcal{H}(t)}$ by particles. 

The point of this replacement is that the resulting ASEP configuration lies in the random state space $\Omega_{\mathcal{H}(t)-\mathcal{P}(t)}$. Assume now that after time $\mathcal{T}$, there is enough time to mix to equilibrium irrespective of what value $\mathcal{P}(t),\mathcal{H}(t)$ take. The bulk of the work in this paper is   to show that this assumption is justified.
On the event $\{\mathcal{P}(t)=L\}\cap \{\mathcal{H}(t)=R\}$, the equilibrium measure then is $\mu_{R-L}$  by Proposition \ref{posrec}. Let us write 
\begin{equation*}
p_{L,R}=\lim_{t\to\infty}\Pb(\{\mathcal{P}(t)=L\}\cap \{\mathcal{H}(t)=R\}), R,L\geq 0.
\end{equation*}
So conditioned on an event with asymptotic  probability $p_{L,R},$  $\delta_{\eta^{1}}S(t)$ should pick $\mu_{R-L}$ as limit measure. Summing over all possible values of $\mathcal{P}(t),\mathcal{H}(t)$ we obtain 
 \begin{equation}\label{Theorem3s}
\lim_{t\to\infty}\delta_{\eta^{1}}S(t)=\sum_{R,L\geq 0}p_{L,R}\mu_{R-L}.
\end{equation}
This statement appears as Theorem \ref{MainPrep} in Section \ref{mainproof}.  The probability $p_{L,R}$ is defined  in \eqref{18}. Once \eqref{Theorem3s} is in place, we can combine the properties of the $p_{L,R}$ and the $\mu_Z$   to obtain Theorem \ref{Main}. 

To understand Theorem \ref{muhat}, start two ASEPs from $\eta^{-\mathrm{step}(Z+1)},\eta^{-\mathrm{step}(Z)},$ coupled via the basic coupling and thus giving rise to  a second class particle $X^{Z}$ which starts from $Z.$  Following the preceding heuristics, the law of $\eta^{2}_{t}\tau_{X(t)}$ should be equal to the law of $\eta^{-\mathrm{step}(R-L+1)}\tau_{X^{R-L}(t)}$ with probability $p_{L,R}$. However,
the law of $\eta^{-\mathrm{step}(R-L+1)}\tau_{X^{R-L}(t)}$ does not depend on $R-L$: The process as seen from $X^{R-L}(t)$ is the same for all $R,L$ (see Proposition \ref{Zprop}), and it converges to $\hat{\mu}$ by Theorem \ref{2.2}. Therefore, $\eta^{2}_{t}\tau_{X(t)}$ converges to $\hat{\mu}$ also.

Finally, as opposed to \cite{N20CMP}, this reasoning  can be used to compute \begin{equation*}\lim_{t\to \infty}\Pb(X(t)\leq i),i\in \Z,\end{equation*} also. It follows in particular from the heuristics  for Theorem \ref{muhat},   that  with probability $p_{L,R},$ $X(t)$ should be equal to $X^{R-L}(t)$. Therefore, we should have 
\begin{equation}
\lim_{t\to\infty}\Pb(X(t)\leq i)=\sum_{L,R\geq 0} p_{L,R}\lim_{t\to\infty}\Pb(X^{R-L}(t)\leq i).
\end{equation}
This statement appears as Theorem \ref{X(t)} in Section \ref{seven}, and it is possible to rederive Theorem \ref{MAIN} from it. 

\section{Results for  $\mathcal{P}(t),\mathcal{H}(t)$  }\label{3}
Here we define the random variables $\mathcal{P}(t),\mathcal{H}(t)$ and give the limiting results we need from them.
Then we define the event $\mathcal{F}_{L,R}^{\delta}$ which will in Proposition \ref{C} of Section \ref{propsec} tell us in which $\Omega_{Z}$  the configuration  $\eta_{t}$ lies after the replacement procedure described in Section \ref{heur}.  The results of the present  section largely come from \cite{N20CMP}, which dealt with $\eta_{t},$ the configuration with a second class particle defined in \eqref{eta},   and we will mostly work with $\eta_t$ to be able to use the results of \cite{N20CMP} without justification.  Note that we can always recover  $\eta^{1}_{t}$ by replacing the second class particle by a first class particle.
Let us first assign a label to the particles of $\eta$ via
\begin{equation}\label{fc}
x_{n}(0)=
\begin{cases}
-n-N(t,M)    &\mathrm{if} \,  n \geq 1 \\
-n+1  &\mathrm{if}\, -N(t,M) +1 \leq n \leq 0.
\end{cases}
\end{equation}
and to the holes of $\eta$ via
\begin{equation}\label{ICH}
H_{n}(0)=
\begin{cases}
n+N(t,M)  \,\,  &\mathrm{for} \,  n \geq 1 \\
n-1\,\, &\mathrm{for}\, -N(t,M) +1 \leq n \leq 0.
\end{cases}
\end{equation}

We define for $0<\chi<\chi^{\prime}<1/2$ 
\begin{equation}\label{PH}
\begin{aligned}
&\mathcal{P}(t)=\sup\{i\in \Z| x_{i}(t-t^{\chi})>-t^{\chi^{\prime}}\}
\\&\mathcal{H}(t)=\sup\{i\in \Z| H_{i}(t-t^{\chi})<t^{\chi^{\prime}}\}.
\end{aligned}
\end{equation}
So $ \mathcal{P}(t), \mathcal{H}(t)$ denote  the label of the leftmost particle (resp. the  rightmost hole) which have reached the position $-t^{\chi^{\prime}}$ (resp. $t^{\chi^{\prime}}$)
by time $t-t^{\chi}.$ The idea behind this definition is that since $\chi<\chi^{\prime},$ and particles and holes have bounded speed, all particles $x_{\mathcal{P}(t)+n},n\geq 1,$ and holes $H_{\mathcal{H}(t)+n},n\geq 1,$ will at time $t$ be too far from the origin to affect $\delta_{\eta^{1}}S(t).$

To state the limit laws of $ \mathcal{P}(t), \mathcal{H}(t),$ we need to introduce the following distribution function.
\begin{defin}[ \cite{TW08b},\cite{GW90}] Let $s \in \R,M \in \Z_{\geq 1}$. We define for $p\in (1/2,1)$  
\begin{equation} \label{def1}
F_{M,p}(s)=\frac{1}{2 \pi i} \oint  \frac{\dx \lambda}{\lambda} \frac{\det(I-\lambda K)}{\prod_{k=0}^{M-1}(1-\lambda (q/p)^{k})}
\end{equation}
where $q=1-p,$ $K= \hat{K}\mathbf{1}_{(-s,\infty)}$  and 
$\hat{K}(z,z^{\prime})=\frac{p}{\sqrt{2 \pi}}e^{-(p^{2}+q^{2})(z^{2}+z^{\prime 2})/4+pqzz^{\prime}}$ and the integral is taken over a counterclockwise oriented contour enclosing the
poles $\lambda=0,\lambda=(p/q)^{k},k=0,\ldots,M-1$.  For $p=1$, we define
\begin{equation*}
F_{M,1}(s)=\Pb\left(\sup_{0=t_{0}<\cdots <t_{M}=1}\sum_{i=0}^{M-1}[B_{i}(t_{i+1})-B_{i}(t_{i})]\leq s    \right),
\end{equation*}
where $B_{i},i=0,\ldots,M-1$ are independent standard Brownian motions. 
Finally, we define for all $p\in (1/2,1]$ 
\begin{equation*}
F_{0,p}(s)=1.
\end{equation*}
\end{defin}

Upon sending $M\to \infty$, we recover the $\GUE$ distribution, as the next result shows.
\begin{prop}[Proposition 2.1 in \cite{N20AAP}]\label{FGUE}

Let $s \in \R.$  Then we have 
\begin{equation}\label{two1}
\lim_{M \to \infty}F_{M,p}\left(\frac{2\sqrt{M}+sM^{-1/6}}{\sqrt{p-q}}\right)=F_{\GUE}(s).
\end{equation}
\end{prop}
Note that the preceding Proposition in particular implies for fixed $s$
\begin{equation}\label{FM0}
\lim_{M \to \infty}F_{M,p}\left(s\right)=0.
\end{equation}

We can now state the result about  the limit distribution of $ \mathcal{P}(t), \mathcal{H}(t)$.
\begin{prop}[Proposition 4.1 in \cite{N20CMP}]\label{convp}
Let $\mathcal{H}(t),\mathcal{P}(t)$ be defined as in \eqref{PH} and recall $C(M)=2\sqrt{\frac{M}{p-q}}$.
We have for  $L\in \Z_{\geq 0}$
\begin{align*}
&\lim_{t\to \infty}\Pb(\mathcal{H}(t)=L)=\lim_{t\to \infty}\Pb(\mathcal{P}(t)=L)=F_{L,p}(C(M))-F_{L+1,p}(C(M))
\\&\lim_{t\to \infty}\Pb(\mathcal{H}(t)<0)=\lim_{t\to \infty}\Pb(\mathcal{P}(t)<0)=0.
\end{align*}
\end{prop}
An important feature of $ \mathcal{P}(t), \mathcal{H}(t)$ is that they decouple asymptotically.
This independence is stated in the following proposition.
\begin{prop}[Proposition 4.2 in \cite{N20CMP}]\label{prodprop}We have for $R,L \in \Z$ 
\begin{equation*}
\begin{aligned}
&\lim_{t\to\infty}\Pb(\{\mathcal{P}(t)=L\}\cap \{\mathcal{H}(t)=R\})
=\lim_{t\to\infty}\Pb(\mathcal{P}(t)=L)\Pb(\mathcal{H}(t)=R).
\end{aligned}
\end{equation*}
\end{prop}
For later usage, we define the probability 
\begin{equation}\label{18}
p_{L,R}:=(F_{L,p}(C(M))-F_{L+1,p}(C(M)))(F_{R,p}(C(M))-F_{R+1,p}(C(M)))
\end{equation} for $L,R\geq 0$ and $p_{L,R}:=0$ if $L<0$ or $R<0,$
so that 
\begin{equation}\label{plr}
p_{L,R}=\lim_{t\to\infty}\Pb(\{\mathcal{P}(t)=L\}\cap \{\mathcal{H}(t)=R\}).
\end{equation}
Furthermore, it is easy to see from the definition and \eqref{FM0} that the $p_{L,R}$  sum up to $1$:
\begin{equation}\label{plr2}
\sum_{R,L\geq 0}p_{L,R}=1.
\end{equation}
Putting together the preceding Propositions \ref{FGUE} - \ref{prodprop} then yields the following theorem.
\begin{tthm}[Theorem 3 in \cite{N20CMP}]\label{Theorem3}
Let $\xi \in \R.$
Recall $p(\xi)=\Pb(\chi_{\GUE}-\chi_{\GUE}^{\prime}\leq \xi),$ where $\chi_{\GUE},\chi_{\GUE}^{\prime}$ are two independent, GUE-distributed random variables. We have
 \begin{equation*}
\lim_{M\to \infty}\lim_{t \to \infty}\Pb\left(\frac{\mathcal{H}(t)-\mathcal{P}(t)}{M^{1/3}} \leq \xi \right)=p(\xi).
\end{equation*}
\end{tthm}

Next we define for $0<\delta<\chi$ the events 
\begin{align*}
&B_L=\{x_{L}(t-t^{\chi})>-t^{\delta}\}\cap\{x_{L+1}(t-t^{\chi})\leq -t^{\chi^{\prime}}\}
\\& D_R=\{H_{R}(t-t^{\chi})<t^{\delta}\}\cap\{H_{R+1}(t-t^{\chi})\geq t^{\chi^{\prime}}\}.
\end{align*}
The point of the events $B_L \cap  D_R$ is that if $B_L, D_R$ happen, we can be sure that the particles and holes $x_{n}(t),n>L, H_{n}(t), n>R$ will play no role
for how the configuration  $\eta^{1}_{t}$  looks like close to the origin: On $B_L \cap D_R$ the $x_{n}(t),n>L, H_{n}(t), n>R$ will be at  distance at least $\mathcal{O}(t^{\chi^{\prime}}-t^{\chi})$ from the origin  and hence they will not be seen by  the measure $\delta_{\eta^{1}}S(t)$ as $t\to \infty$. On the other hand, since $\delta<\chi,$ we can show that   $H_{R}$ and $x_{L}$ are at time $t$ close to the origin and affect $\delta_{\eta^{1}}S(t)$ as $t\to \infty$.
Define
\begin{equation}\label{FLR}
\mathcal{F}_{L,R}^{\delta}=B_{L}\cap D_{R} \cap \{   |X(t-t^{\chi})|\leq t^{\delta}\}.
\end{equation}
Note that $\mathcal{F}_{L,R}^{\delta} \subseteq \{\mathcal{P}(t)=L\}\cap \{\mathcal{H}(t)=R\}$. As the next Proposition shows, $\mathcal{F}_{L,R}^{\delta}$ has  asymptotically the same probability as  $ \{\mathcal{P}(t)=L\}\cap \{\mathcal{H}(t)=R\}$.

\begin{prop}\label{plrflr}We have \begin{equation}
\lim_{t\to\infty}\Pb(\mathcal{F}_{L,R}^{\delta})= p_{L,R}.
\end{equation}
\end{prop}
\begin{proof}
Note first that \begin{equation}\label{incl}
B_L \subseteq \{\mathcal{P}(t)=L\}, \quad 
D_R \subseteq \{\mathcal{H}(t)=R\}.
\end{equation}
Now it is shown on \cite[page 616]{N20CMP} that
\begin{equation*}
\lim_{t\to\infty}(\Pb(\{\mathcal{P}(t)=L\}\setminus B_L)+\Pb(\{\mathcal{H}(t)=R\}\setminus D_R))=0.
\end{equation*}
Furthermore, by \cite[Proposition 5.2]{N20CMP} we have 
\begin{equation}\label{5.2}
\lim_{t\to \infty}\Pb(|X(t-t^{\chi})|>t^{\delta})=0
\end{equation}
and hence using Proposition \ref{prodprop} and \eqref{plr} we get
 \begin{align*}
\lim_{t\to\infty}\Pb(\mathcal{F}_{L,R}^{\delta})= \lim_{t\to\infty}\Pb(\{\mathcal{P}(t)=L\}\cap \{\mathcal{H}(t)=R\})   &=   \lim_{t\to\infty}\Pb(\mathcal{P}(t)=L)\Pb(\mathcal{H}(t)=R)   \\&   = p_{L,R}.
\end{align*}
\end{proof}

Finally, the following simple proposition will be used repeatedly to justify that the $t \to \infty$ limit may be taken inside a series.
\begin{prop}\label{D}
For every $\varepsilon>0$ there is an integer $D>0$ such that 
\begin{equation*}
\lim_{t\to\infty}\sum_{0\leq R,L\leq D}\Pb(\mathcal{F}_{L,R}^{\delta})\geq 1-\varepsilon.
\end{equation*}
\end{prop}

\begin{proof}According to \eqref{FM0}, for every $M,\varepsilon>0$  there is a $D>0$ such that $(1-F_{D+1,p}(C(M)))>1-\varepsilon/2$. Furthermore, by Proposition
\ref{plrflr} we have $\lim_{t\to\infty}\Pb(\mathcal{F}_{L,R}^{\delta})=p_{L,R}$. Hence 
\begin{equation*}
\sum_{0\leq R,L\leq D}\lim_{t\to\infty}\Pb(\mathcal{F}_{L,R}^{\delta})=(1-F_{D+1,p}(C(M)))(1-F_{D+1,p}(C(M)))\geq 1-\varepsilon.
\end{equation*}
\end{proof}

\section{Reduction to positive recurrent ASEPs}\label{propsec}

The aim of this section is to show that on the event $\mathcal{F}_{L,R}^{\delta}$,   $\eta^{1}_{t}$ equals, within a large enough  neighborhood of the origin, a particle configuration in $\Omega_{R-L}$ which is close to its equilibrium $\mu_{R-L}$.
This is done in three steps. First we show in Proposition \ref{E} that $\eta^{1}_{t}$ equals (within a  neighborhood of the origin) a configuration  $\tilde{\eta}^{1}_{t}$ which lies in $\cup_{Z\in \Z}\Omega_{Z}$.
Then we introduce a particle configuration  $\hat{\eta}^{Z}$  which lies in  $\Omega_{-Z}$  and show that $\hat{\eta}^{Z}_{t}$ is close to equilibrium. Finally, in Proposition \ref{C} we show  on the event $\mathcal{F}_{L,R}^{\delta}$ that $\tilde{\eta}^{1}_{t}$ equals $\hat{\eta}^{L-R}_{t}$. Recall that we always have constants $0<\delta<\chi<\chi^{\prime}<1/2$.

We define 
\begin{equation*}
\tilde{\eta}_{t-t^{\chi}}(j)=\mathbf{1}_{\{|j|\leq t^{\delta}\}}\eta_{t-t^{\chi}}(j)+\mathbf{1}_{\{j>t^{\delta}\}}.
\end{equation*}
Then  $(\tilde{\eta}_{\ell},\ell\geq t-t^{\chi})$ is the ASEP which starts at time $t-t^{\chi} $ from $\tilde{\eta}_{t-t^{\chi}}$. If $|X(t-t^{\chi})|\leq t^{\delta}$, the process  $(\tilde{\eta}_{\ell},\ell\geq t-t^{\chi})$  has a second class particle at position $X(t-t^{\chi}),$  recall that by \eqref{5.2}, $\Pb(|X(t-t^{\chi})|>t^{\delta})$ goes to $0$. The processes $(\tilde{\eta}^{1}_{\ell},\tilde{\eta}^{2}_{\ell},\ell\geq t-t^{\chi})$ are defined analogously: If $|X(t-t^{\chi})|\leq t^{\delta}$, then  $\tilde{\eta}^{1}_{t-t^{\chi}}$ is obtained from $\tilde{\eta}_{t-t^{\chi}}$ by replacing  the second class  particle by a first class particle, whereas in $ \tilde{\eta}^{2}_{\ell},$ the second class particle is replaced by a hole (if $|X(t-t^{\chi})|> t^{\delta},$  all three processes coincide). 

By \eqref{5.2}, we thus have that  $\tilde{\eta}^{1}_{t-t^{\chi}}(j),\tilde{\eta}^{2}_{t-t^{\chi}}(j)$ disagree at position $X(t-t^{\chi})$ with probability going to $1$. 
 We can then define using the basic coupling  the second class particle 
\begin{equation}\label{X(t)2}
\tilde{X}(t):=\sum_{j\in \Z}j\mathbf{1}_{\{\tilde{\eta}^{1}_{t}(j)\neq \tilde{\eta}^{2}_{t}(j)\}},
\end{equation}
which asymptotically equals $X(t)$, as the next proposition shows.
\begin{prop}\label{tilde=nottilde}
We have 
\begin{equation*}
\lim_{t\to\infty}\Pb(X(t)=\tilde{X}(t))=1.
\end{equation*}
\end{prop}
\begin{proof}
Let $\varepsilon>0.$ By Proposition \ref{D}, there is  a $D>0$ such that  
\begin{equation*}
\lim_{t\to\infty}\Pb(X(t)\neq\tilde{X}(t))\leq \varepsilon+\lim_{t\to\infty}\sum_{0\leq R,L\leq D}\Pb(\{X(t)\neq\tilde{X}(t)\}\cap \mathcal{F}_{L,R}^{\delta}),
\end{equation*}
and the r.h.s. equals $\varepsilon$ by \cite[Proposition 5.3]{N20CMP}. Since $\varepsilon>0$ is arbitrary, the result follows. 
\end{proof}
Next we show that $\eta_{t}$ equals $\tilde{\eta}_{t}$  in a large neighborhood of the origin.
\begin{prop}\label{E}
We have that 
\begin{equation}
\lim_{t\to\infty}\Pb(\tilde{\eta}_{t}(j)=\eta_{t}(j) \mathrm{\,for\, all \,  }j\in\{- t^{\chi^{\prime}}/2,\ldots,  t^{\chi^{\prime}}/2\})=1.
\end{equation}
\end{prop}
\begin{proof}Consider the event that neither $(\tilde{\eta}_{\ell},\ell\geq t-t^{\chi})$ nor $(\eta_{\ell},\ell\geq t-t^{\chi})$ have a jump at the sites $\pm  t^{\chi^{\prime}}/2$ during 
$[t-t^{\chi},t]$ . This event can be written as the intersection of
\begin{align*} 
&\tilde{\mathcal{E}}_{t}=\{\mathrm{\,for\, all \, }\ell \in [t-t^{\chi},t] \mathrm{\,and\,}i\in \{1,2\}, \tilde{\eta}_{\ell}((-1)^{i} t^{\chi^{\prime}}/2)= \tilde{\eta}_{t-t^{\chi}}((-1)^{i} t^{\chi^{\prime}}/2)\}
\\&\mathcal{E}_{t}=\{\mathrm{\,for\, all \, }\ell \in [t-t^{\chi},t] \mathrm{\,and\,}i\in \{1,2\}, \eta_{\ell}((-1)^{i} t^{\chi^{\prime}}/2)= \eta_{t-t^{\chi}}((-1)^{i} t^{\chi^{\prime}}/2)\}.
\end{align*}

It was shown in  \cite[Equation (57)]{N20CMP}  that 
\begin{equation}\label{007}
\lim_{t\to\infty}(\Pb(\mathcal{F}_{L,R}^{\delta}\cap \mathcal{E}_{t}^{c})+\Pb(\mathcal{F}_{L,R}^{\delta}\cap \tilde{\mathcal{E}}_{t}^{c}))=0.
\end{equation}
Furthermore, in  \cite[page 620]{N20CMP} it was shown that for all $R,L\geq 0$ we have
\begin{equation*}
\mathcal{F}_{L,R}^{\delta} \cap \tilde{\mathcal{E}}_{t}\cap\mathcal{E}_{t}\subseteq \{ \tilde{\eta}_{t}(j)=\eta_{t}(j) \mathrm{\,for\, all \,  }j\in\{- t^{\chi^{\prime}}/2,\ldots,  t^{\chi^{\prime}}/2\}\},
\end{equation*}
which implies 
\begin{equation*}
\bigcup_{R,L\geq 0}\mathcal{F}_{L,R}^{\delta} \cap \tilde{\mathcal{E}}_{t}\cap\mathcal{E}_{t}\subseteq \{ \tilde{\eta}_{t}(j)=\eta_{t}(j)  \mathrm{\,for\, all \,  }j\in\{- t^{\chi^{\prime}}/2,\ldots,  t^{\chi^{\prime}}/2\}\}.
\end{equation*}
Since $\mathcal{F}_{L,R}^{\delta},R,L\geq 0 $ are disjoint events,  we thus have 
\begin{equation}\label{**}
\lim_{t\to\infty}\Pb(\tilde{\eta}_{t}(j)=\eta_{t}(j)  \mathrm{\,for\, all \,  }j\in\{- t^{\chi^{\prime}}/2,\ldots,  t^{\chi^{\prime}}/2\})\geq \lim_{t\to\infty}\sum_{R,L\geq 0}\Pb(\mathcal{F}_{L,R}^{\delta} \cap \tilde{\mathcal{E}}_{t}\cap\mathcal{E}_{t}) .
\end{equation}
Now  by Proposition \ref{D}, for every $\varepsilon >0$ there is a $D>0$ such that 
\begin{align*}
\lim_{t\to\infty}\sum_{R,L\geq 0}\Pb(\mathcal{F}_{L,R}^{\delta} \cap \tilde{\mathcal{E}}_{t}\cap\mathcal{E}_{t})&\geq   \sum_{D\geq R,L\geq 0}\lim_{t\to\infty}\Pb(\mathcal{F}_{L,R}^{\delta} \cap \tilde{\mathcal{E}}_{t}\cap\mathcal{E}_{t})
\\&\geq \sum_{D\geq R,L\geq 0}\lim_{t\to\infty}\Pb(\mathcal{F}_{L,R}^{\delta} )=1-\varepsilon,
\end{align*}
where in the second equality we also used \eqref{007}. Since $\varepsilon>0$ is arbitrary, this finishes the proof.

\end{proof}

Next we define for $0<\delta<\chi<\chi^{\prime},Z\in \Z$ and $t>0$ with  $t^{\delta}>|Z|$   the configuration
\begin{equation}\label{etaZ}
\eta^{Z}:=\mathbf{1}_{\{-t^{\delta},\ldots,Z\}}+\mathbf{1}_{\Z_{>t^{\delta}}}\in \Omega_{-Z}.
\end{equation}
The following proposition shows that the time interval $[0,t^{\chi}]$ is long enough for the ASEP started from $\eta^{Z}$ to converge to equilibrium.
A key tool to prove this is the upper bound on the mixing time of (finite state space) ASEP given in \cite{BBHM} (later  the cutoff phenomenon was proven in  \cite{LL19}, and recently, the cutoff profile was obtained \cite{BN20}, however the results of \cite{BBHM} suffice for our purposes).

\begin{prop}\label{Zdelta}Let $\eta^{Z}$ be given by \eqref{etaZ}. Then we have in the sense of weak convergence of measures for all $\chi>\delta$ 
\begin{equation}
\lim_{t\to\infty}\delta_{\eta^{Z}}S(t^{\chi})=\mu_{Z}.
\end{equation}
\end{prop}
\begin{proof}
The idea of the proof is to show that, under the basic coupling,  with probability going to $1$, the ASEP started from $\eta^{Z}$ coalesces with the ASEP started from $\eta^{-\mathrm{step}(-Z)}$ before time $t^{\chi}$.

We  define 

\begin{equation*}
I_{0}=\mathbf{1}_{\{-t^{\delta}-2Z-1,\ldots,-Z-1\}}+\mathbf{1}_{\Z_{>t^{\delta}}},
\end{equation*} and let $(I_{\ell},\ell\geq 0)$  be the ASEP started from $I_{0}$.
Recalling the partial order \eqref{order}, we have  $I_{0}\preceq \eta^{Z}$.

Now we define a particle configuration  $I^{Z+t^{\delta}+1}$ via  
\begin{equation}\label{mossel}
1-I_{0}(j)=I^{Z+t^{\delta}+1}(j+Z), \quad j\in \Z.
\end{equation}
$I^{Z+t^{\delta}+1}$ is exactly the particle configuration defined in  equation (4) of the paper   \cite{BBHM}.
Consider the hitting time
\begin{equation*}
\mathfrak{H}(I_0)=\inf\{\ell:  I_{\ell}=\eta^{-\mathrm{step}(-Z)}\}.
\end{equation*}
By \eqref{mossel}, Theorem 1.9 of \cite{BBHM}  gives that for every $\varepsilon>0$ 
\begin{equation*}
\lim_{t\to\infty}\Pb(\mathfrak{H}(I_0)\geq t^{\delta+\varepsilon})=0.
\end{equation*}
We choose $\varepsilon  >0 $ so that $\delta+\varepsilon<\chi$. We note that we have the inclusion

\begin{equation}\label{H0}
\{\mathfrak{H}(I_0)\leq \ell \}\subseteq \{\eta^{Z}_{\ell}= \eta^{-\mathrm{step}(-Z)}_{\ell}\}
\end{equation}
because of the relations 
\begin{equation*}
 \eta^{-\mathrm{step}(-Z)} \succeq \eta^{-\mathrm{step}(-Z)}_{\mathfrak{H}(I_0)}\succeq \eta^{Z}_{\mathfrak{H}(I_0)}\succeq I_{\mathfrak{H}(I_0)}= \eta^{-\mathrm{step}(-Z)},
\end{equation*}
and thus $ \eta^{-\mathrm{step}(-Z)}_{\mathfrak{H}(I_0)}= \eta^{Z}_{\mathfrak{H}(I_0)}$ and \eqref{H0} holds. Hence, with $f_{A}$ from Proposition \ref{convprop}, we have since 
$\delta+\varepsilon<\chi$
\begin{equation}
\begin{aligned}\label{19}
\lim_{t \to \infty} \delta_{\eta^{Z}}S(t^{\chi})(f_A)&=\lim_{t \to \infty} \Pb(\{\eta^{Z}_{t^{\chi}}(i)=1  \mathrm{\, for \, all \,}i\in A\} \cap \{ \mathfrak{H}(I_0)\leq t^{\delta+\varepsilon}\})
\\&=\lim_{t \to \infty} \Pb(\eta^{-\mathrm{step}(-Z)}_{t^{\chi}}(i)=1  \mathrm{\, for \, all \,}i\in A).
\end{aligned}
\end{equation}
Applying Proposition \ref{posrec}  to the last line of \eqref{19} finishes the proof by Proposition \ref{convprop}.
\end{proof}

In the following, we will  consider the ASEP which starts \textit{at time} $t-t^{\chi}$   from $\eta^{Z}$  with $\delta<\chi$.  To make this clear in our notation, we set 
\begin{equation}\label{etaZhat}
\hat{\eta}^{Z}_{t-t^{\chi}}:=\eta^{Z}
\end{equation}
so that $(\hat{\eta}^{Z}_{\ell},\ell \geq t-t^{\chi})$ is the process which starts at time $t-t^{\chi}$ from $\eta^{Z}$.    In particular, the law of the configuration $\hat{\eta}_{t}^{Z}$ equals 
$\delta_{\eta^{Z}}S(t^{\chi})$.  We will couple the process $(\hat{\eta}^{Z}_{\ell},\ell \geq t-t^{\chi})$ with all other appearing ASEPs via the basic coupling. A simple but crucial observation  is then that 
since $(\hat{\eta}^{Z}_{\ell},\ell \geq t-t^{\chi})$ starts from a deterministic initial configuration at time $t-t^{\chi}$, the process  $(\hat{\eta}^{Z}_{\ell},\ell \geq t-t^{\chi})$ is independent of all events which solely depend on what happens during $[0,t-t^{\chi}]$.  In particular,  $\hat{\eta}^{Z}_{t}$ is independent from the event  $\mathcal{F}_{L,R}^{\delta}$ defined in \eqref{FLR}. Likewise, we will  consider the ASEP which starts \textit{at time} $t-t^{\chi}$   from $\eta^{-\mathrm{step}(Z)}$  and set \begin{equation}\label{etaZhat2}
\hat{    \eta}^{-\mathrm{step}(Z)}_{t-t^{\chi}}     :=\eta^{-\mathrm{step}(Z)}.
\end{equation}

\begin{prop}\label{C}
Recall the event $\mathcal{F}_{L,R}^{\delta}$ from \eqref{FLR}, the probability $p_{L,R}$ from \eqref{18} and $\hat{\eta}^{Z}_{t-t^{\chi}},\hat{    \eta}^{-\mathrm{step}(Z)}_{t-t^{\chi}} $ from \eqref{etaZhat},\eqref{etaZhat2}. We have that 
\begin{align*}
&\lim_{t\to\infty}\Pb\left(\mathcal{F}_{L,R}^{\delta} \cap\{ \hat{\eta}_{t}^{L-R}=\hat{    \eta}^{-\mathrm{step}(R-L)}_{t} =\tilde{\eta}^{1}_{t}\}\cap  \{ \hat{\eta}_{t}^{L-R-1}=\hat{    \eta}^{-\mathrm{step}(R-L+1)}_{t} =\tilde{\eta}^{2}_{t}\}\right)\\&=\lim_{t\to\infty}\Pb(\mathcal{F}_{L,R}^{\delta})=p_{L,R}.
\end{align*}
\end{prop}
\begin{proof}

According to (60) of \cite{N20CMP}, we have 
\begin{equation}\label{bound}
\mathcal{F}_{L,R}^{\delta} \subseteq\{\hat{\eta}^{L-R}_{t-t^{\chi}}\preceq \tilde{\eta}^{1}_{t-t^{\chi}}, \hat{\eta}^{L-R-1}_{t-t^{\chi}}\preceq \tilde{\eta}^{2}_{t-t^{\chi}}\}.
\end{equation}
We can now reason as in the proof of Proposition \ref{Zdelta}: As was shown there, the processes  $(\hat{\eta}^{L-R}_{\ell},\ell \geq t-t^{\chi}), (\hat{\eta}^{L-R-1}_{\ell},\ell \geq t-t^{\chi})$ reach the states 
$\eta^{-\mathrm{step}(R-L)}$ and $ \eta^{-\mathrm{step}(R-L+1)}$ during $[t-t^{\chi},t]$ with probability going to $1$, and hence on $\mathcal{F}_{L,R}^{\delta}$  the processes  $(\hat{\eta}^{L-R}_{\ell},\ell \geq t-t^{\chi}), (\hat{\eta}^{L-R-1}_{\ell},\ell \geq t-t^{\chi})$
coalesce with the processes $ (\hat{    \eta}^{-\mathrm{step}(R-L)}_{\ell},\ell \geq t-t^{\chi}), (\hat{    \eta}^{-\mathrm{step}(R-L+1)}_{\ell},\ell \geq t-t^{\chi})$
before time $t$ with probability going to $1.$ Consequently, the processes $(\tilde{\eta}^{1}_{\ell},\ell \geq t-t^{\chi}), (\tilde{\eta}^{2}_{\ell},\ell \geq t-t^{\chi})$ coalesce with the processes $ (\hat{    \eta}^{-\mathrm{step}(R-L)}_{\ell},\ell \geq t-t^{\chi}), (\hat{    \eta}^{-\mathrm{step}(R-L+1)}_{\ell},\ell \geq t-t^{\chi})$
before time $t$ with probability going to $1$ also.
\end{proof}

Finally, we define a second class particle $
\hat{X}^{Z}(t)$ which starts at time $t-t^{\chi}$ from position $Z\in \Z$, and has initially only particles to its right, and only holes to its left. In terms of the particle configurations \eqref{etaZhat2}, we have
\begin{equation}
 \hat{X}^{Z}(t)=\sum_{j\in \Z}j\mathbf{1}_{\{  \hat{\eta}^{-\mathrm{step}(Z+1)}_{t}(j)\neq   \hat{\eta}^{-\mathrm{step}(Z)}_{t}(j)\}}.
\end{equation}
As corollary from the previous proposition we obtain the following.
\begin{prop}\label{G}
\begin{equation}\label{aless}
\lim_{t\to\infty}\Pb(\mathcal{F}_{L,R}^{\delta} \cap \{ \hat{X}^{R-L}(t)=\tilde{X}(t)\})=\lim_{t\to\infty}\Pb\left(\mathcal{F}_{L,R}^{\delta}\right).
\end{equation}
\end{prop}
\begin{proof} Trivially, the l.h.s. in \eqref{aless} is bounded from above by the r.h.s.
It thus  suffices to note that \begin{equation*}
\mathcal{F}_{L,R}^{\delta} \cap\{ \hat{    \eta}^{-\mathrm{step}(R-L)}_{t} =\tilde{\eta}^{1}_{t}\}\cap  \{ \hat{    \eta}^{-\mathrm{step}(R-L+1)}_{t} =\tilde{\eta}^{2}_{t}\}\subseteq \mathcal{F}_{L,R}^{\delta}\cap  \{ \hat{X}^{R-L}(t)=\tilde{X}(t)\}
\end{equation*}
and apply Proposition \ref{C}.
\end{proof}

\section{Proof of Theorem \ref{Main}}\label{mainproof}
Here we  first prove Theorem \ref{MainPrep} and then derive Theorem \ref{Main} from it.
The following proposition collects a few of the properties of the measures  $\mu_Z$ that will be used throughout this section.
\begin{prop} \label{muprop}Let $Z,n\in \Z$ and  $A\subset \Z$ be finite.  Denote $i=\min(A)$. There are constants $C_{1},C_{2}>0$  which depend on $p$ but not on $Z,n,A$ such that we have 
\begin{enumerate}[label=(\roman*)]
\item $\quad1- C_{1}e^{-C_{2}(i-Z)}\leq  \mu_{Z}(f_{A})\leq C_{1}e^{-C_{2}(Z-i-1)} $
\item $\quad  \mu_{Z}(f_A)\geq \mu_{Z+1}(f_A)$
\item $  \quad  \mu_{Z}\tau_{n}=\mu_{Z-n}$
\item $  \quad   \mu_{0}(f_{\{n\}})=1-\mu_{0}(f_{\{-n-1\}})$.
\end{enumerate}

\end{prop}
\begin{proof} Throughout the proof, we will be using Proposition \ref{posrec}.
For (i),  let us deal with  the limit $Z\to +\infty$ first.						
Consider  ASEP with reversed step  initial data $x_{-n}^{\mathrm{-step(Z)}}(0)=n+Z,n\geq 0.$ Denote furthermore  by $ H_{0}^{\mathrm{-step}(Z)}(t)$ be the position of the rightmost hole of $\eta_{t}^{-\mathrm{step}(Z)}$. It follows from \cite[Proposition 3.1]{N20AAP} that   there are constants   $C_{1},C_{2}>0$ (which depend on $p$) such that for $R\in \Z_{\geq 1}$  we have
\begin{align}\label{trivial}
&\Pb\left(  x_{0}^{\mathrm{-step}(Z)}(t)<Z-R\right)\leq C_{1}e^{-C_{2}R},
\end{align}
by choosing $C_{1}>1,$ \eqref{trivial} trivially extends to all $R\in \Z$.
Applying particle-hole duality, this implies
\begin{align*}
\Pb\left(  H_{0}^{\mathrm{-step}(Z)}(t)<Z+R\right)\geq 1-  C_{1}e^{-C_{2}R}.
\end{align*}
Recall  $i=\min(A)$. Then we have 
\begin{align*}
\mu_{Z}(f_{A})\leq \mu_{Z}(f_{\{i\}})\leq\lim_{t\to\infty} \Pb\left(  x_{0}^{\mathrm{-step}(Z)}(t)<Z+i+1-Z\right)\leq C_{1}e^{-C_{2}(Z-i-1)},
\end{align*}
and analogously we have 
\begin{align*}
\mu_{Z}(f_{A}) \geq \lim_{t\to\infty} \Pb\left(  H_{0}^{\mathrm{-step}(Z)}(t)<Z+i -Z\right)\geq 1- C_{1}e^{-C_{2}(i-Z)}.
\end{align*}

To prove (ii), note that under the basic coupling, coordinatewise we have  $\eta^{-\mathrm{step}(Z)}_{t}\geq \eta^{-\mathrm{step}(Z+1)}_{t}$ and hence
\begin{align*}
 \mu_{Z}(f_A)&=\lim_{t\to\infty}\Pb(\eta^{-\mathrm{step}(Z)}_{t}(i)=1 \mathrm{\, for \,all\,}i\in A )\\&\geq \lim_{t\to\infty}\Pb(\eta^{-\mathrm{step}(Z+1)}_{t}(i)=1\mathrm{\, for \,all\,}i\in A )= \mu_{Z+1}(f_A).
\end{align*}
For (iii), we  note that we have $\delta_{\zeta \tau_{n}}S(t)=\delta_{\zeta}\tau_{n}S(t)=\delta_{\zeta}S(t)\tau_{n}$ and hence for $\zeta\in \Omega_{Z}$ we have 
\begin{align}
\mu_{Z}\tau_{n}=\lim_{t\to\infty}\delta_{\zeta}S(t)\tau_{n}=\lim_{t\to\infty}\delta_{\zeta\tau_{n}}S(t)=\mu_{Z-n}.
\end{align}

For (iv), set $\eta^{\otimes}_{t}(j)=\mathbf{1}_{\Z}(-j)-\eta_{t}^{-\mathrm{step}(0)}(-j).$ This is an ASEP starting from $ \mathbf{1}_{\Z_{\geq 1}}$,
applying  Proposition \ref{posrec} twice thus yields
\begin{equation}
\mu_{1}(f_{\{j\}})=\lim_{t\to\infty} \Pb(\eta^{\otimes}_{t}(j)=1)= 1-\lim_{t\to\infty}\Pb(\eta_{t}^{-\mathrm{step}(0)}(-j)=1)=1-\mu_{0}(f_{\{-j\}})
\end{equation}
Since $\mu_{1}(f_{\{j\}})=\mu_{0}(f_{\{j-1\}})$ by (iii),  (iv) follows by setting $j-1=n$.

\end{proof}
We can now obtain the $t\to\infty$ limit law of $\eta^{1}_{t}$. A heuristical derivation of the following theorem was given in Section \ref{heur}. 

\begin{tthm}\label{MainPrep}We have in the sense of weak convergence of measures
\begin{equation*}
\lim_{t\to\infty}\delta_{\eta^{1}}S(t)=\sum_{R,L\geq 0}p_{L,R}\mu_{R-L}.
\end{equation*}
\end{tthm}
\begin{rem}\label{rem}
In the entirety of the paper, we have considered $C(M)=2(M/(p-q))^{1/2}$ which in particular goes to $+\infty$ as $M\to+\infty$.
It is however also possible to replace $C(M)$ by a $\widetilde{C}(M)$ that goes to $-\infty$ as $M\to\infty$.  For such a choice,  we have for $L,R\geq 0$ 
\begin{equation*}
p_{L,R}=(F_{L,p}(\widetilde{C}(M))-F_{L+1,p}(\widetilde{C}(M)))(F_{R,p}(\widetilde{C}(M))-F_{R+1,p}(\widetilde{C}(M)))
\end{equation*}
so that consequently
$
\lim_{M\to \infty} p_{0,0}=1$ and $ \lim_{M\to \infty} p_{L,R}=0, (L,R)\neq (0,0).
$ In particular, we then  get by (iii) of Proposition \ref{muprop}
\begin{equation}
\lim_{M\to\infty}\left(\sum_{R,L\geq 0}p_{L,R}\mu_{R-L}\right)\tau_{M^{1/3}\xi}=\lim_{M\to\infty}\mu_{-M^{1/3}\xi}=\begin{cases}\mu_{0} & \xi=0\\\nu_{0} &\xi< 0 \\ \nu_{1} & \xi >0.
\end{cases}
\end{equation}
In this sense, $\lim_{t\to\infty}\delta_{\eta^{1}}S(t)\tau_{M^{1/3}\xi}$ is an infinite linear combination of invariant measures which  interpolates between   the mixture $(1-p(\xi))\nu_{0}+p(\xi)\nu_{1}$ and  one of the single equilibriums $\mu_{0},\nu_{0},\nu_{1}$ depending on the sign of $\xi$.

\end{rem}

\begin{proof}[Proof of Theorem \ref{MainPrep}]
By Proposition \ref{convprop}, we have to  prove
\begin{equation}
\lim_{t \to \infty} \delta_{\eta^{1}}S(t) (f_A)=\lim_{t \to \infty}\Pb(\eta_{t}^{1}(i)=1  \mathrm{\, for \, all \,}i\in A)=\sum_{R,L\geq 0}p_{L,R}\mu_{R-L}(f_A).
\end{equation}
This is clear for  $A=\emptyset$, i.e. if $f_A$ is constant $1$, since the $p_{L,R}$ sum up to $1$.
For arbitrary finite $A,$ note that it follows from  Proposition \ref{E} 
\begin{align*}
\lim_{t\to \infty}\Pb(\eta^{1}_{t}(i)=1  \mathrm{\, for \, all \,}i\in A)=\lim_{t\to \infty}\Pb(\tilde{\eta}^{1}_{t}(i)=1  \mathrm{\, for \, all \,}i\in A),
\end{align*}
simply because if $\eta_{t}(i)=\tilde{\eta}_{t}(i),$ then also $\eta_{t}^{1}(i)=\tilde{\eta}^{1}_{t}(i)$.

Since $\lim_{t \to \infty } \Pb(\cup_{R,L\geq 0}\mathcal{F}_{L,R}^{\delta})=1$  and the $\mathcal{F}_{L,R}^{\delta}$ are pairwise disjoint, we obtain 
\begin{align*}
\lim_{t\to \infty}\Pb(\tilde{\eta}^{1}_{t}(i)=1  \mathrm{\, for \, all \,}i\in A)=\lim_{t\to \infty}\sum_{R,L\geq 0}\Pb(\mathcal{F}_{L,R}^{\delta}\cap \{\tilde{\eta}^{1}_{t}(i)=1  \mathrm{\, for \, all \,}i\in A\}).
\end{align*}
Let now $\varepsilon>0.$ By Proposition \ref{D}, there is a positive integer $D$ such that 
\begin{equation}\label{DDD}
\begin{aligned}
&\lim_{t\to \infty}\sum_{R,L\geq 0}\Pb(\mathcal{F}_{L,R}^{\delta}\cap \{\tilde{\eta}^{1}_{t}(i)=1  \mathrm{\, for \, all \,}i\in A\})\\&\leq \varepsilon+\lim_{t\to \infty} \sum_{D \geq R,L\geq 0}\Pb(\mathcal{F}_{L,R}^{\delta}\cap \{\tilde{\eta}^{1}_{t}(i)=1  \mathrm{\, for \, all \,}i\in A\}).
\end{aligned}
\end{equation}
Now by Proposition \ref{C}, we have that 
\begin{align*}
\lim_{t\to\infty}\Pb(\mathcal{F}_{L,R}^{\delta}\cap \{\tilde{\eta}^{1}_{t}(i)=1  \mathrm{\, for \, all \,}i\in A\})=\lim_{t\to\infty}\Pb(\mathcal{F}_{L,R}^{\delta}\cap \{\hat{\eta}^{R-L}_{t}(i)=1  \mathrm{\, for \, all \,}i\in A\}).
\end{align*}
Note that  $\hat{\eta}^{R-L}_{t}$ is independent of $ \mathcal{F}_{L,R}^{\delta}$ by construction. Furthermore, the law of $\hat{\eta}^{R-L}_{t}$ converges to $\mu_{R-L}$ by Proposition \ref{Zdelta},  and thus
\begin{align*}
\lim_{t\to\infty}\Pb(\mathcal{F}_{L,R}^{\delta}\cap \{\hat{\eta}^{R-L}_{t}(i)=1  \mathrm{\, for \, all \,}i\in A\})&=\lim_{t\to\infty}\Pb(\mathcal{F}_{L,R}^{\delta})\Pb( \hat{\eta}^{R-L}_{t}(i)=1  \mathrm{\, for \, all \,}i\in A)
\\&=p_{L,R}\mu_{R-L}(f_A).\end{align*}
Hence we obtain from \eqref{DDD} that 
\begin{equation}\label{DDDD}
\begin{aligned}
&\lim_{t\to \infty}\sum_{R,L\geq 0}\Pb(\mathcal{F}_{L,R}^{\delta}\cap \{\tilde{\eta}^{1}_{t}(i)=1  \mathrm{\, for \, all \,}i\in A\})\leq \varepsilon+\sum_{R,L\geq 0}p_{L,R}\mu_{R-L}(f_A)\end{aligned}
\end{equation}
Likewise we obtain the lower bound 
\begin{equation}\label{DD}
\begin{aligned}
&\lim_{t\to \infty}\sum_{R,L\geq 0}\Pb(\mathcal{F}_{L,R}^{\delta}\cap \{\tilde{\eta}^{1}_{t}(i)=1  \mathrm{\, for \, all \,}i\in A\})\\&\geq \lim_{t\to \infty} \sum_{D \geq R,L\geq 0}\Pb(\mathcal{F}_{L,R}^{\delta}\cap \{\tilde{\eta}^{1}_{t}(i)=1  
\mathrm{\, for \, all \,}i\in A\})\\&=\sum_{D \geq R,L\geq 0}p_{L,R}\mu_{R-L}(f_A)\geq -\varepsilon+\sum_{ R,L\geq 0}p_{L,R}\mu_{R-L}(f_A),
\end{aligned}
\end{equation}
where for the last inequality we used Propositions \ref{plrflr} and \ref{D}.
This finishes the proof, since $\varepsilon>0$ is arbitrary.
\end{proof}
Now we can prove Theorem \ref{Main}.

\begin{proof}[Proof of Theorem \ref{Main}]

Recall that by part (iii) of Proposition \ref{muprop} we have    $\mu_Z \tau_{n}=\mu_{Z-n}$ . Using Theorem \ref{MainPrep} and Proposition \ref{convprop}, we thus have to prove for all finite $A\subset \Z$
\begin{equation}\label{1}
\lim_{M\to\infty} \sum_{R,L\geq 0}p_{L,R}\mu_{R-L-M^{1/3}\xi}(f_A)=(1-p(\xi))\nu_{0}(f_A)+p(\xi)\nu_{1}(f_A).
\end{equation}
The r.h.s. of \eqref{1} equals $1$ for $A=\emptyset$ and  $p(\xi)$ otherwise. Equation \eqref{1} is clearly  true for $A=\emptyset$ since the $p_{L,R}$ sum up to $1$, so we assume $A\neq \emptyset$ in the following.

1. Case: $(R,L)\in \mathcal{S}_{1}:=\{(R,L)\in \Z_{0}^{2}: R-L \geq  M^{1/3}\xi+ M^{1/4}\}$. Then, using part (ii)  of Proposition \ref{muprop} we may bound 
\begin{equation}
\sum_{(R,L)\in \mathcal{S}_{1}}p_{L,R}\mu_{R-L-M^{1/3}\xi}(f_A) \leq \mu_{M^{1/4}}(f_A)  \sum_{(R,L)\in \mathcal{S}_{1}}p_{L,R}
\end{equation}
and the r.h.s. goes to zero as $M\to \infty$  by part (i) of Proposition \ref{muprop}.   

2. Case: $(R,L)\in \mathcal{S}_{2}:=\{(R,L)\in \Z_{0}^{2}: R-L \leq  M^{1/3}\xi- M^{1/4}\}$. Then we obtain 
\begin{equation*}
\sum_{(R,L)\in \mathcal{S}_{2}}p_{L,R}\mu_{R-L-M^{1/3}\xi}(f_A) \geq \mu_{-M^{1/4}}(f_A)  \sum_{(R,L)\in \mathcal{S}_{2}}p_{L,R}.
\end{equation*}
 By part (i) of Proposition \ref{muprop}, we have $\lim_{M\to\infty} \mu_{-M^{1/4}}(f_A)=1$. Combining this with \eqref{plr} and Theorem \ref{Theorem3}, we obtain 
\begin{equation}
\lim_{M\to\infty} \mu_{-M^{1/4}}(f_A)  \sum_{(R,L)\in \mathcal{S}_{2}}p_{L,R}\geq \lim_{M\to \infty}\lim_{t\to \infty}\Pb\left(\frac{\mathcal{H}(t)-\mathcal{P}(t)}{M^{1/3}}\leq \xi- M^{-1/12}\right)=p(\xi).
 \end{equation}
 On the other hand, we have 
 \begin{equation*}
\lim_{M\to\infty} \sum_{(R,L)\in \mathcal{S}_{2}}p_{L,R}\mu_{R-L-M^{1/3}\xi}(f_A) \leq \lim_{M\to\infty}   \sum_{(R,L)\in \mathcal{S}_{2}}p_{L,R}=p(\xi).
\end{equation*}
In total, this yields 
 \begin{equation*}
\lim_{M\to\infty} \sum_{(R,L)\in \mathcal{S}_{2}}p_{L,R}\mu_{R-L-M^{1/3}\xi}(f_A) =p(\xi).
\end{equation*}
 3. Case: $(R,L)\in \mathcal{S}_{3}=\{(R,L)\in \Z_{0}^{2}: M^{1/3}\xi- M^{1/4}\leq R-L \leq  M^{1/3}\xi+M^{1/4}\}$. Note  that from  \eqref{plr}  we get  \begin{equation}
   \sum_{(R,L)\in \mathcal{S}_{3}}p_{L,R}=\lim_{t\to \infty}\Pb\left(-M^{-1/12}\leq \frac{\mathcal{H}(t)-\mathcal{P}(t)}{M^{1/3}}-\xi \leq  M^{-1/12}\right),
 \end{equation}
 and the r.h.s. converges to zero as $M \to \infty$ by Theorem  \ref{Theorem3}. This finishes the proof.

\end{proof}

\section{Results for the second class particle}\label{seven}

Here we prove Theorem \ref{muhat}. Furthermore, we show in Corollary \ref{Xcor} that the density of particles to the right (resp. left) of $X(t)$ approaches $1$ (resp. $0$) exponentially fast. Finally, we also prove the $t\to\infty$ limit law of $X(t)$ in Theorem \ref{X(t)}.  We start with an observation needed for Theorem \ref{muhat} that was already mentioned in  Section \ref{heur}. Consider  two ASEPs starting from $\eta^{-\mathrm{step}(Z+1)},\eta^{-\mathrm{step}(Z)}$ and  coupled via the basic coupling, and denote by $X^{Z}(t)$ the position of the induced   second class particle which starts at position $Z$.

\begin{prop}\label{Zprop}
The law   of $\eta^{-\mathrm{step}(Z+1)}_{t}\tau_{X^{Z}(t)}$ does not depend on $Z,$ and   it converges weakly to the measure $\hat{\mu}$ from Theorem \ref{2.2} with $\lambda=1,\rho=0$. 
\end{prop}
\begin{proof}
Define $\xi:=\eta^{-\mathrm{step}(1)}\tau_{-Z},$  so that $(\xi_{\ell},\ell\geq 0)$ is an ASEP starting from $\eta^{-\mathrm{step}(Z+1)}$. Setting $\tilde{X}^{Z}(t)=X^{0}(t)+Z$ we see that  $(\xi_{\ell}\tau_{\tilde{X}^{Z}(\ell)},\ell\geq 0)$ is a version of $(\eta^{-\mathrm{step}(Z+1)}_{\ell}\tau_{X^{Z}(\ell)},\ell \geq 0). $ In particular, $\xi_{t}\tau_{\tilde{X}^{Z}(t)}$has the same law as  $\eta^{-\mathrm{step}(Z+1)}_{t}\tau_{X^{Z}(t)}.$ Since   $\xi_{t}\tau_{\tilde{X}^{Z}(t)}=\eta^{-\mathrm{step}(1)}_{t}\tau_{-Z}\tau_{X^{0}(t)+Z}=\eta^{-\mathrm{step}(1)}_{t}\tau_{X^{0}(t)}$, this shows the first part of the proposition. For the second part, note that the law of $\eta^{-\mathrm{step}(1)}_{t}\tau_{X^{0}(t)}$ is exactly $\hat{\nu}_{0,1}\hat{S}(t)$ from Theorem \ref{2.2}, which gives the desired convergence.
\end{proof}
Now we can prove Theorem \ref{muhat}.
\begin{proof}[Proof of Theorem \ref{muhat}]

Combining  Propositions \ref{tilde=nottilde} and \ref{E}, we get for any finite subset $A$ of $\Z$ 

\begin{equation*}
\lim_{t\to\infty}\Pb\left(\eta_{t}^{2}(X(t)+i)=1 \mathrm{\,for\,all\,} i\in A\right)=\lim_{t\to\infty}\Pb\left(\tilde{\eta}^{2}_{t}(\tilde{X}(t)+i)=1 \mathrm{\,for\,all\,} i\in A\right).
\end{equation*}
Intersecting with the union of the disjoint events  $\cup_{R,L\geq 0}\mathcal{F}^{\delta}_{L,R}$ yields 
\begin{align}\nonumber
&\lim_{t\to\infty}\Pb\left(\tilde{\eta}^{2}_{t}(\tilde{X}(t)+i)=1 \mathrm{\,for\,all\,} i\in A\right)\\&=\label{series}\lim_{t\to\infty}\sum_{R,L\geq 0}\Pb(\mathcal{F}_{L,R}^{\delta}\cap \{\tilde{\eta}^{2}_{t}(\tilde{X}(t)+i)=1 \mathrm{\,for\,all\,} i\in A\}).
\end{align}
A truncation argument identical  to the one given in  \eqref{DDD},\eqref{DDDD}, \eqref{DD} shows  that for any $\varepsilon>0$ we have 
\begin{align*}
&-\varepsilon+ \sum_{R,L\geq 0}\lim_{t\to\infty}\Pb(\mathcal{F}_{L,R}^{\delta}\cap \{\tilde{\eta}^{2}_{t}(\tilde{X}(t)+i)=1 \mathrm{\,for\,all\,} i\in A\})
\\&\leq \lim_{t\to\infty}\sum_{R,L\geq 0}\Pb(\mathcal{F}_{L,R}^{\delta}\cap \{\tilde{\eta}^{2}_{t}(\tilde{X}(t)+i)=1 \mathrm{\,for\,all\,} i\in A\})\\&\leq\varepsilon +\sum_{R,L\geq 0}\lim_{t\to\infty}\Pb(\mathcal{F}_{L,R}^{\delta}\cap \{\tilde{\eta}^{2}_{t}(\tilde{X}(t)+i)=1 \mathrm{\,for\,all\,} i\in A\}),
\end{align*}
i.e. we may take the $\lim_{t \to \infty}$inside the series \eqref{series}.
Using Propositions \ref{C},  \ref{G} we have 
\begin{align*}
&\lim_{t\to\infty}\Pb(\mathcal{F}_{L,R}^{\delta}\cap \{\tilde{\eta}^{2}_{t}(\tilde{X}(t)+i)=1 \mathrm{\,for\,all\,} i\in A\})\\&=\lim_{t\to\infty}\Pb(\mathcal{F}_{L,R}^{\delta}\cap \{\hat{\eta}^{-\mathrm{step}(R-L+1)}_{t}(\hat{X}^{R-L}(t)+i)=1 \mathrm{\,for\,all\,} i\in A\}).
\end{align*}
Furthermore, $\hat{\eta}^{-\mathrm{step}(R-L+1)}_{t}(\hat{X}^{R-L}(t)+\cdot)$ is by construction independent from $\mathcal{F}_{L,R}^{\delta}$. Combining this with Proposition \ref{Zprop}, we get 
\begin{align*}
&\lim_{t\to\infty}\Pb(\mathcal{F}_{L,R}^{\delta}\cap \{\hat{\eta}^{-\mathrm{step}(R-L+1)}_{t}(\hat{X}^{R-L}(t)+i)=1 \mathrm{\,for\,all\,} i\in A\})
\\&=\lim_{t\to\infty}\Pb(\mathcal{F}_{L,R}^{\delta})\Pb( \{\hat{\eta}^{-\mathrm{step}(R-L+1)}_{t}(\hat{X}^{R-L}(t)+i)=1 \mathrm{\,for\,all\,} i\in A\})
\\&=p_{L,R}\hat{\mu}(f_{A}).
\end{align*}
So in  total, using again that the $p_{L,R}$ sum up to one, we obtain 
\begin{equation*}
\lim_{t\to\infty}\Pb\left(\eta_{t}^{2}(X(t)+i)=1 \mathrm{\,for\,all\,} i\in A\right)=
\sum_{R,L\geq 0}p_{L,R}\hat{\mu}(f_{A})=\hat{\mu}(f_{A}).
\end{equation*}

\end{proof}

Next we show that the densities $1,0$ are reached exponentially fast from $X(t)$.
For this, we need the limit law of $X^{Z}(t).$  
\begin{prop}[Proposition 1.2 in \cite{N20CMP}]\label{yours}
Consider ASEP started from $ \eta^{-\mathrm{step}(Z+1)}$ and a second class particle  $X^{Z}$ starting from $Z$. 
Then for $i \in \Z$  \begin{equation}\label{eq}
\lim_{t \to \infty}\Pb(X^{Z}(t)=i)=\mu_{0}(f_{\{i-Z+1\}})-\mu_{0}(f_{\{i-Z\}}).
\end{equation}
\end{prop}
We now have all the ingredients to show the following:
\begin{cor}\label{Xcor}
There are constants $C_{1},C_{2}>0$ such that   for all  $n \geq 1$ 
\begin{align*}
&\lim_{t\to\infty}\Pb(\eta^{2}_{t}(X(t)+n)=1)\geq 1- C_{1}e^{-C_{2}n}
\\&\lim_{t\to\infty}\Pb(\eta^{2}_{t}(X(t)-n)=0)\geq 1- C_{1}e^{-C_{2}n}.
\end{align*}
\end{cor}
\begin{proof} We can compute 
\begin{align*}
&\lim_{t \to \infty}\Pb(\eta^{-\mathrm{step}(1)}_{t}(X^{0}(t)+n)=1)
\\&\geq \lim_{t\to\infty} \Pb(\{\eta^{-\mathrm{step}(1)}_{t}(j)=1\mathrm{\, for \, all\,} j=n/2,\ldots, 3n/2\} \cap   \{|X^{0}(t) |\leq n/2\})
\\&\geq \mu_{1}(f_{\{n/2,\ldots,3n/2\}}) -\lim_{t\to\infty}\Pb( \{|X^{0}(t) |> n/2\})
\\&\geq 1- C_{1}e^{-C_{2}n}- \mu_{0}(f_{\{-2-n/2\}})-\mu_{0}(f_{\{-n/2\}})\geq 1- 3C_{1}e^{-C_{2}n},
\end{align*}
where for the third inequality we used Proposition \ref{yours} and (i),(iv) from Proposition \ref{muprop}, and the last inequality used again (i) from Proposition \ref{muprop}.
This proves the first claim of the  corollary, since 
\begin{equation*}
\lim_{t \to \infty}\Pb(\eta^{-\mathrm{step}(1)}_{t}(X^{0}(t)+n)=1)=\lim_{t \to \infty}\Pb(\eta^{2}_{t}(X(t)+n)=1)
\end{equation*}
by Theorem \ref{muhat}. The second claim follows by the particle-hole duality.
\end{proof}

Finally, we are able to obtain the $t\to\infty$ limit law of $X(t)$.
\begin{tthm}\label{X(t)}
We have for $i\in \Z$
\begin{equation}
\lim_{t\to\infty}\Pb(X(t)\leq i)=\sum_{L,R\geq 0} p_{L,R}\mu_{0}\left(f_{\{i+L-R+1\}}\right).
\end{equation}
\end{tthm}
\begin{proof}
This is very similar to the proof of Theorem \ref{muhat}: We intersect with the event $\cup_{R,L\geq 0}\mathcal{F}^{\delta}_{L,R},$ then we use that on each $\mathcal{F}_{L,R}^{\delta}$ we can replace $X(t)$ by  $\hat{X}^{R-L}(t),$ which is independent 
from $\mathcal{F}_{L,R}^{\delta}$. Doing this yields
\begin{align}
\lim_{t\to\infty}\Pb(X(t)\leq i)&=\lim_{t\to\infty}\sum_{L,R\geq 0} \Pb(\mathcal{F}_{L,R}^{\delta}\cap\{\hat{X}^{R-L}(t)\leq i\})
\\&=\sum_{L,R\geq 0} \lim_{t\to\infty}\Pb(\mathcal{F}_{L,R}^{\delta})\Pb(\hat{X}^{R-L}(t)\leq i)
\\&=\sum_{L,R\geq 0} p_{L,R}\lim_{t\to\infty}\Pb(X^{R-L}(t^{\chi})\leq i).
\end{align}
The result follows from Proposition \ref{yours}.
\end{proof}
We remark that, by the same argument used in the proof of Theorem \ref{Main}, it is possible to rederive   Theorem \ref{MAIN}
from Theorem \ref{X(t)}.

\bibliography{Biblio}{}
\bibliographystyle{plain}

\end{document}